\documentclass[11pt]{article}

\usepackage[margin=1in]{geometry}
\usepackage{amsmath,amssymb,amsthm,mathtools}
\usepackage[colorlinks=true,citecolor=blue,linkcolor=blue,urlcolor=blue]{hyperref}
\hypersetup{
  pdftitle={Strong Convergence of I.I.D. Non-Hermitian Random Matrices under a Fourth-Moment Hypothesis},
  pdfauthor={Yanjin Xiang and Zhihua Zhang},
  pdfkeywords={strong convergence, non-Hermitian random matrices, circular elements, finite fourth moment}
}

\newtheorem{theorem}{Theorem}[section]
\newtheorem{proposition}[theorem]{Proposition}
\newtheorem{lemma}[theorem]{Lemma}
\newtheorem{corollary}[theorem]{Corollary}
\theoremstyle{definition}
\newtheorem{definition}[theorem]{Definition}
\theoremstyle{remark}
\newtheorem{remark}[theorem]{Remark}

\DeclareMathOperator{\Tr}{Tr}
\DeclareMathOperator{\E}{\mathbb E}
\DeclareMathOperator{\spn}{sp}
\newcommand{\trn}{\operatorname{tr}}
\newcommand{\one}{\mathbf 1}

\title{Strong Convergence for a General Class of Random Matrix Models}
\author{%
  Yanjin Xiang and Zhihua Zhang\\
  School of Mathematical Sciences, Peking University\\
  {\small\texttt{yjxiang@stu.pku.edu.cn; zhzhang@math.pku.edu.cn}}
}
\date{\today}

\begin{document}

\maketitle

\begin{abstract}
Let \(X_{1,n},\ldots,X_{d,n}\) be \(n\times n\) random matrices built from
independent i.i.d. entry arrays, with centered entries,
normalized by \(n^{-1/2}\).  We prove that, if every entry law has finite
fourth moment, then this tuple converges almost surely strongly in
\(*\)-distribution to a free circular family with the matching variances.
Equivalently, normalized traces and operator norms converge for every fixed
noncommutative \(*\)-polynomial, including polynomials with fixed matrix
coefficients.  No assumption is imposed on the
pseudo-variances of the complex entries.  The bounded-entry argument applies
the spectrum and moment universality estimates of Brailovskaya and van Handel
to all self-adjoint linear pencils.  The matching Gaussian pencils are reduced
to independent Wigner matrices and identified by Anderson's strong convergence
theorem.  A fixed-level centered truncation, followed by the Bai--Yin norm
bound, transfers the result to finite fourth moments.
\end{abstract}



\section{Introduction}

Let \(X_{1,n},\ldots,X_{d,n}\) be independent non-Hermitian random matrices
with centered independent and identically distributed (i.i.d.) entries of variances of order \(n^{-1}\).  In the
Gaussian case, Voiculescu's asymptotic freeness theorem identifies their joint
normalized \(*\)-moment limit as a free circular family
\cite{voiculescu-1991}.  Strong convergence asks for the strictly finer
conclusion
\[
 \|Q(X_{1,n},\ldots,X_{d,n},X_{1,n}^*,\ldots,X_{d,n}^*)\|
 \longrightarrow
 \|Q(c_1,\ldots,c_d,c_1^*,\ldots,c_d^*)\|
\]
for every fixed noncommutative \(*\)-polynomial \(Q\).  In particular, strong
convergence rules out norm outliers for every polynomial in the tuple.

This conclusion is different from the classical circular law.  The latter
describes the empirical eigenvalue distribution of a single nonnormal matrix
\cite{bai-1997-circular,tao-vu-krishnapur-2010}, whereas strong
\(*\)-convergence controls normalized \(*\)-moments and operator norms of
every fixed polynomial in the whole tuple.  Neither statement is a formal
substitute for the other: in particular, strong \(*\)-convergence alone does
not give Brown-measure convergence for an arbitrary nonnormal polynomial,
because that passage requires additional control of small singular values.
The spectral consequences proved here therefore concern self-adjoint
polynomials and the singular-value spectra of arbitrary polynomials.

Strong convergence for Gaussian matrix polynomials was first proved in the
complex case in~\cite{haagerup-thorbjornsen-2005} and then in the real and
symplectic settings in~\cite{schultz-2005}.
Anderson proved strong convergence for independent Wigner matrices under a
fourth-moment condition, including matrix-valued polynomials
\cite{anderson-2013-polynomial}; see also Male~\cite{male-2012} for strong
convergence in the presence of additional matrices.
Belinschi et al.\ proved spectral confinement and
outlier results for non-Hermitian polynomials in i.i.d. matrices
and deterministic matrices~\cite{belinschi-bordenave-capitaine-cebron-2021}.
Under their balanced-entry assumption, all normalized real and imaginary
coordinates are independent and identically distributed.  Without
deterministic matrices, their theorem gives the upper spectral inclusion for
self-adjoint polynomials in that subclass and, with standard \(*\)-moment
convergence, the corresponding scalar-coefficient strong limit.  Their
deterministic and finite-rank results are complementary to the present scope.

More recently, Bandeira et al.\ extended this phenomenon of strong asymptotic freeness to an extremely general class of Gaussian random matrices~\cite{bvh}. Furthermore, 
nonasymptotic universality estimates were established by comparing the
spectrum and fixed moments of a general sum of independent self-adjoint random
matrices with those of the Gaussian matrix having the same covariance
\cite{brailovskaya-van-handel-2024}.  These results also yield a
broad strong asymptotic freeness theorem for self-adjoint matrix families.

The remaining issue addressed here is that the Hermitian coordinates
\[
 S_{\kappa,n}=\frac{X_{\kappa,n}+X_{\kappa,n}^*}{\sqrt2},
 \qquad
 T_{\kappa,n}=\frac{X_{\kappa,n}-X_{\kappa,n}^*}{i\sqrt2}
\]
do not need to be independent.  We therefore work one self-adjoint linear pencil at
a time.  For bounded entries, every such pencil is a sum of independent
self-adjoint entry matrices.  Its maximal summand norm is
\(O(n^{-1/2})\), its variance parameter is \(O(1)\), and its weak variance
parameter is \(O(n^{-1/2})\).  Spectrum and moment universality consequently
compare it, with summable error probabilities, to the matching Gaussian
pencil.

This pencilwise step is essential even though the available comparison
theorem is stated for self-adjoint random matrices.  Applying a strong
asymptotic freeness theorem directly to the family of Hermitian coordinates
would require independence that is absent when the complex entry has nonzero
pseudo-variance.  Instead, the covariance of each complete pencil is matched
before any strong limit is taken.  Matrix-valued linearization then
reconstructs all joint polynomial norms from these scalar-in-\(n\)
self-adjoint comparisons.

Each Gaussian entry may have arbitrary pseudo-variance.  Its real covariance
matrix nevertheless permits a representation as a complex linear combination
of two independent real Gaussian variables.  The associated
matrix is a linear combination of two real Ginibre matrices, and each real
Ginibre matrix is the complex combination of two independent Wigner matrices.
Anderson's theorem therefore identifies the strong limit.  A short cumulant
calculation shows that the pseudo-variance disappears at leading order and
that the limit is circular.

For the lower norm bound, spectrum comparison alone is insufficient.  We use
moment universality and the concentration inequality for even moment roots
from~\cite{brailovskaya-van-handel-2024}.  The only additional integrability
input is a uniform \(L^2\) bound for normalized traces of powers of the
Gaussian pencils; we prove it directly by Wick expansion.  This yields all
even pencil moments on one probability-one event.  The standard
linearization criterion then supplies joint \(*\)-moment convergence and the
upper norm bound, while faithfulness of the free trace supplies the reverse
norm inequality.

Finally, boundedness of the entry laws is removed at a fixed truncation level
\(K\).  The centered tail at that level is again one fixed i.i.d. law.  The
Bai--Yin estimate, applied separately to its real and imaginary parts, bounds
the tail matrix norm by a constant times its standard deviation, which tends
to zero as \(K\to\infty\).  This order of limits avoids any appeal to a
changing-law triangular-array version of the Bai--Yin theorem.

The proof is a specialization and synthesis of three deep inputs rather than
a replacement for them: the universality estimates of Brailovskaya and van
Handel, Anderson's strong Wigner theorem, and the Bai--Yin norm bound.  The
purpose of the paper is to verify in full detail that these inputs compose for
non-Hermitian i.i.d. matrices beyond the balanced complex subclass above,
including arbitrary complex pseudo-variance (and hence real or degenerate
laws), matrix-valued linearization, one common almost-sure event, and the
finite-fourth truncation passage.

No moment above the fourth is used.  At the level of a single coordinate, the
theorem gives \(\|X_{\kappa,n}\|\to2\sigma_\kappa\), the Bai--Yin edge
scale.  Its content is substantially stronger than this coordinate estimate:
the same probability-one event controls every fixed mixed polynomial in all
colors and their adjoints, including nonnormal polynomials and arbitrary
complex coefficients.

For powers and fixed products, a direct combinatorial proof under the same
fourth-moment assumption was given in
\cite{xiang-chen-zhang-2026}.  That argument identifies the Fuss--Catalan
edge explicitly and controls repeated-label collisions by a defect-sensitive
high-moment enumeration.  The present theorem recovers those norm limits as
special cases, but by a different route: it first proves joint strong
convergence and then evaluates an arbitrary fixed polynomial at the limiting
circular family.  Conversely, the earlier enumeration supplies information
specific to product words that is not part of the present universality
argument.

The paper is organized as follows.  Section~\ref{sec:background} reviews the
free-probability, spectral, linearization, covariance, and Gaussian terminology
used in the proof.  Section~\ref{sec:model} states the model and main theorem.
Section~\ref{sec:inputs} records the precise comparison and linearization
inputs.  Section~\ref{sec:bounded} proves strong convergence for bounded
entries.  Section~\ref{sec:tail} performs the finite-fourth truncation
transfer.  Section~\ref{sec:discussion} discusses the interpretation,
limitations, and possible extensions of the result.

\section{Preliminaries}\label{sec:background}

This section recalls the basic notions used in the statement and proof.  It
introduces no additional hypothesis or intermediate theorem.

\subsection{Tracial \texorpdfstring{\(C^*\)}{C-star}-probability spaces}

A tracial \(C^*\)-probability space is a pair \((\mathcal A,\tau)\), where
\(\mathcal A\) is a unital \(C^*\)-algebra and
\(\tau:\mathcal A\to\mathbb C\) is a positive unital linear functional such
that
\[
 \tau(ab)=\tau(ba)
 \quad\text{for all }a,b\in\mathcal A.
\]
The trace is \emph{faithful} if
\(\tau(a^*a)=0\) implies \(a=0\).  For a self-adjoint element \(h\), the trace
determines a compactly supported probability measure \(\mu_h\) by
\[
 \tau(p(h))=\int_{\mathbb R}p(t)\,d\mu_h(t)
\]
for every polynomial \(p\).  If \(\tau\) is faithful, the support of
\(\mu_h\) is the whole \(C^*\)-spectrum of \(h\).  In particular,
\[
 \|a\|
 =\lim_{r\to\infty}
   \tau\bigl((a^*a)^r\bigr)^{1/(2r)}.
\]
Unital subalgebras
\(\mathcal A_1,\ldots,\mathcal A_\ell\subseteq\mathcal A\) are
\emph{free} if
\[
 \tau(a_1\cdots a_m)=0
\]
whenever each \(a_j\) is centered, consecutive factors come from different
subalgebras, and \(m\geq1\).  A centered self-adjoint element \(s\) is
semicircular with variance \(\sigma^2\) if its spectral distribution is the
semicircle law on \([-2\sigma,2\sigma]\).  Circular elements, defined in
Section~\ref{sec:model}, are the non-Hermitian analogues obtained from two
free semicircular coordinates.

\subsection{Joint \texorpdfstring{\(*\)}{star}-distribution and strong
convergence}

For a tuple \(a=(a_1,\ldots,a_d)\), its joint \(*\)-distribution consists of
the numbers
\[
 \tau\bigl(Q(a_1,\ldots,a_d,a_1^*,\ldots,a_d^*)\bigr)
\]
as \(Q\) ranges over noncommutative \(*\)-polynomials.  Convergence in
\(*\)-distribution means convergence of all these quantities.  For random
matrices, the trace used here is the normalized trace
\(\trn_n=n^{-1}\Tr\).

Strong convergence adds
\[
 \|Q(a^{(n)},(a^{(n)})^*)\|
 \longrightarrow
 \|Q(a,a^*)\|
\]
for every fixed \(Q\).  The word \emph{fixed} is important throughout this
paper: the polynomial, its degree, and its coefficients do not vary with
\(n\).  Likewise, when coefficients lie in \(M_q(\mathbb C)\), the
amplification size \(q\) is fixed before \(n\to\infty\).  The statement that
one probability-one event works for every fixed \(q\) does not assert
uniformity for a sequence \(q=q_n\).

The Introduction explains the distinction between strong
\(*\)-convergence and the circular law.  In particular, strong
\(*\)-convergence does not by itself imply eigenvalue-distribution or
Brown-measure convergence for nonnormal polynomials, for which small
singular-value control is an additional issue.

\subsection{Spectral conventions}

For a self-adjoint element \(h\), write \(\spn(h)\subseteq\mathbb R\) for its
spectrum.  If \(K\) and \(L\) are nonempty compact subsets of a metric space,
their Hausdorff distance is
\[
 d_{\rm H}(K,L)
 =\max\left\{
   \sup_{x\in K}\operatorname{dist}(x,L),
   \sup_{y\in L}\operatorname{dist}(y,K)
 \right\}.
\]
Thus \(d_{\rm H}(K_n,K)\to0\) means both that every point of \(K_n\) is close
to \(K\) and that every point of \(K\) is approximated by points of \(K_n\).

For an arbitrary element \(a\), singular-value information is obtained from
the spectrum of the positive element \(a^*a\) by the square-root map; the
formal definition is given in Section~\ref{sec:model}.  Passing from \(a\) to
\(a^*a\) is useful because the latter is self-adjoint even when \(a\) is
nonnormal.

\subsection{Matrix coefficients, self-adjointification, and linear pencils}

If
\[
 P\in M_q(\mathbb C)\otimes
 \mathbb C\langle x_1,\ldots,x_d,x_1^*,\ldots,x_d^*\rangle,
\]
then \(P(X_n,X_n^*)\) is an element of \(M_q\otimes M_n\), equipped with the
operator norm and the normalized trace \(\trn_q\otimes\trn_n\).  Matrix
coefficients are needed because linearization replaces a polynomial by a
larger affine expression.

The first elementary step is self-adjointification:
\[
 \widehat P=
 \begin{pmatrix}
 0&P\\
 P^*&0
 \end{pmatrix}.
\]
Then \(\widehat P=\widehat P^*\) and
\(\|\widehat P\|=\|P\|\).  Thus a norm question for an arbitrary polynomial
can be embedded into a spectral question for a self-adjoint polynomial.

A self-adjoint linear pencil in self-adjoint variables
\(y_1,\ldots,y_m\) is an expression
\[
 L(y)=A_0\otimes\one+\sum_{j=1}^m A_j\otimes y_j,
 \qquad A_0,\ldots,A_m\in M_q(\mathbb C)_{\rm sa}.
\]
Standard linearization represents the resolvent of a self-adjoint polynomial
as a corner, or equivalently a Schur complement, of the resolvent of a
fixed-size self-adjoint pencil; see~\cite{helton-mai-speicher-2018}.  This is
the standard reason that
matrix-valued self-adjoint pencils appear in strong-convergence arguments.

\subsection{Hermitian coordinates and covariance matching}

Every matrix \(X\) can be recovered from its Hermitian coordinates
\[
 S=\frac{X+X^*}{\sqrt2},
 \qquad
 T=\frac{X-X^*}{i\sqrt2},
 \qquad
 X=\frac{S+iT}{\sqrt2}.
\]
Even when the entries of \(X\) are independent, \(S\) and \(T\) need not be
independent.

For a centered complex random variable \(\xi\), write
\[
 \sigma^2=\E|\xi|^2,
 \qquad
 m=\E[\xi^2].
\]
The number \(m\) is often called the \emph{pseudo-variance}.  It records the
imbalance and correlation between the real and imaginary parts:
\[
 \operatorname{Cov}
 \begin{pmatrix}\operatorname{Re}\xi\\ \operatorname{Im}\xi\end{pmatrix}
 =
 \frac12
 \begin{pmatrix}
  \sigma^2+\operatorname{Re}m&\operatorname{Im}m\\
  \operatorname{Im}m&\sigma^2-\operatorname{Re}m
 \end{pmatrix}.
\]
This covariance matrix is positive semidefinite because
\(|m|\leq\sigma^2\).  Thus specifying both \(\sigma^2\) and \(m\) is
equivalent to specifying the full covariance of the two-dimensional real
vector \((\operatorname{Re}\xi,\operatorname{Im}\xi)\).  The corresponding
Gaussian realization is constructed later in
Lemma~\ref{lem:gaussian-realization}.

\subsection{Gaussian matrix terminology and Wick's formula}

A real Ginibre matrix with variance \(n^{-1}\) is an \(n\times n\) matrix
whose entries are independent real \(N(0,n^{-1})\) variables.  A Hermitian
Wigner matrix has independent centered entries on and above the diagonal
(subject to Hermitian symmetry), with off-diagonal variance \(n^{-1}\);
the diagonal variance may be treated separately.  These conventions explain
the normalization used when the matching Gaussian matrices are decomposed in
the proof.

If \(g_1,\ldots,g_{2e}\) are jointly centered real Gaussian variables, Wick's
formula states that
\[
 \E[g_1\cdots g_{2e}]
 =\sum_{\pi\in\mathcal P_2(2e)}
   \prod_{\{r,s\}\in\pi}\E[g_r g_s],
\]
where \(\mathcal P_2(2e)\) is the set of pair partitions of
\(\{1,\ldots,2e\}\).  Odd centered Gaussian moments vanish.  In the matrix
expansions below, each covariance factor in a Wick pairing identifies the
corresponding matrix indices; this is the elementary input behind the
Gaussian trace-moment count.

\section{Model and main result}\label{sec:model}

We write \(\trn_N=N^{-1}\Tr\) for the normalized trace on \(M_N(\mathbb C)\).
All matrix norms are operator norms.

\begin{definition}[The i.i.d.\ entry model]\label{def:iid-entry-model}
Fix \(d\in\mathbb N\).  Given a probability space, let
\[
 \bigl(\xi^{(\kappa)}_{ij}\bigr)_{i,j\geq1},
 \qquad 1\leq\kappa\leq d,
\]
be mutually independent infinite arrays.  Within each color \(\kappa\), the
entries are i.i.d.\ copies of a complex random variable \(\xi^{(\kappa)}\)
satisfying
\[
 \E\xi^{(\kappa)}=0,
 \qquad
 \E|\xi^{(\kappa)}|^2=\sigma_\kappa^2,
 \qquad
 \E|\xi^{(\kappa)}|^4<\infty.
\]
For \(n\geq1\), set
\[
 X_{\kappa,n}
 =\frac1{\sqrt n}
   \bigl(\xi^{(\kappa)}_{ij}\bigr)_{1\leq i,j\leq n}.
\]
Thus, for each color, the unnormalized entry blocks are nested upper-left
corners of one infinite array; the matrix at size \(n\) is obtained by applying
the normalization \(n^{-1/2}\) to the corresponding corner.  This coupling
across \(n\) is used for every almost-sure statement below.  No restriction is
placed on
\(\E[(\xi^{(\kappa)})^2]\).
\end{definition}

Let \((\mathcal A,\tau)\) be a tracial \(C^*\)-probability space with faithful
trace.  A circular element of variance \(\sigma^2\) is an element
\(c=(s_1+is_2)/\sqrt2\), where \(s_1,s_2\) are free centered semicircular
elements with \(\tau(s_1^2)=\tau(s_2^2)=\sigma^2\).  A family is circular if
all of its real and imaginary semicircular coordinates are free.

\begin{definition}[Strong convergence]
A random matrix tuple \(Y_n\) converges almost surely strongly to a tuple
\(y\) if, on one probability-one event, for
every noncommutative \(*\)-polynomial \(Q\),
\[
 \trn_n Q(Y_n,Y_n^*)\longrightarrow\tau(Q(y,y^*)),
 \qquad
 \|Q(Y_n,Y_n^*)\|\longrightarrow\|Q(y,y^*)\|.
\]
\end{definition}

\begin{theorem}[Finite-fourth strong circular limit]\label{thm:main}
Let \(X_{1,n},\ldots,X_{d,n}\) satisfy
Definition~\ref{def:iid-entry-model}.  Let \(c_1,\ldots,c_d\) be a free
circular family in a tracial \(C^*\)-probability space with faithful trace,
with
\[
 \tau(c_\kappa c_\lambda^*)=\delta_{\kappa\lambda}\sigma_\kappa^2.
\]
Then \((X_{1,n},\ldots,X_{d,n})\) converges almost surely strongly in
\(*\)-distribution to \((c_1,\ldots,c_d)\).
Moreover, on the same probability-one event, for every \(q\geq1\) and every
\[
 Q\in M_q(\mathbb C)\otimes
 \mathbb C\langle x_1,\ldots,x_d,x_1^*,\ldots,x_d^*\rangle,
\]
one has
\begin{align*}
 \|Q(X_n,X_n^*)\|&\longrightarrow\|Q(c,c^*)\|,\\
 (\trn_q\otimes\trn_n)Q(X_n,X_n^*)
 &\longrightarrow(\trn_q\otimes\tau)Q(c,c^*).
\end{align*}
\end{theorem}

\begin{corollary}\label{cor:homogeneous}
For every fixed noncommutative polynomial \(P\), in particular for every fixed
homogeneous polynomial with arbitrary complex coefficients,
\[
 \|P(X_{1,n},\ldots,X_{d,n})\|
 \longrightarrow
 \|P(c_1,\ldots,c_d)\|
 \qquad\text{almost surely}.
\]
\end{corollary}

\begin{corollary}[Fuss--Catalan product edge]\label{cor:fuss-catalan-edge}
Suppose \(\sigma_1=\cdots=\sigma_d=\sigma\), fix \(k\geq1\), and let
\(i_1,\ldots,i_k\in\{1,\ldots,d\}\), with repetitions allowed.  Then
\[
 \|X_{i_1,n}\cdots X_{i_k,n}\|
 \longrightarrow
 \sigma^k\gamma_k
 \qquad\text{almost surely},
 \qquad
 \gamma_k=\sqrt{\frac{(k+1)^{k+1}}{k^k}}.
\]
\end{corollary}

\begin{proof}
Theorem~\ref{thm:main} reduces the limit to
\(\|c_{i_1}\cdots c_{i_k}\|\).  For arbitrary fixed label words with
repetitions, the identity
\[
 \|c_{i_1}\cdots c_{i_k}\|=\sigma^k\gamma_k
\]
follows from the Gaussian specialization of
\cite[Theorem~1.3]{xiang-chen-zhang-2026} together with Gaussian strong
convergence.
\end{proof}

\begin{remark}
No assumption is made on
\(m_\kappa:=\E[(\xi^{(\kappa)})^2]\).  In particular, the theorem covers real
entry laws, for which \(m_\kappa=\sigma_\kappa^2\), as well as genuinely complex
and degenerate two-dimensional laws.
\end{remark}

For an element \(a\) of a \(C^*\)-algebra, write
\[
 \operatorname{sv}(a)
 =\{\sqrt{\lambda}:\lambda\in\spn(a^*a)\}
\]
for its singular-value spectrum.  For a finite matrix this is the set of its
singular values, with multiplicities suppressed.

\begin{corollary}[Spectral consequences]\label{cor:spectral-consequences}
On the probability-one event of Theorem~\ref{thm:main}, the following hold.
\begin{enumerate}
\item[(1)] If \(Q=Q^*\) is a fixed self-adjoint \(*\)-polynomial, then
\[
 d_{\rm H}\bigl(\spn(Q(X_n,X_n^*)),
                    \spn(Q(c,c^*))\bigr)\longrightarrow0.
\]
\item[(2)] For every fixed \(*\)-polynomial \(Q\),
\[
 d_{\rm H}\bigl(\operatorname{sv}(Q(X_n,X_n^*)),
                    \operatorname{sv}(Q(c,c^*))\bigr)\longrightarrow0.
\]
\item[(3)] If \(Q(c,c^*)\) is invertible, then \(Q(X_n,X_n^*)\) is invertible
for all sufficiently large \(n\), and
\[
 \|Q(X_n,X_n^*)^{-1}\|\longrightarrow\|Q(c,c^*)^{-1}\|.
\]
\item[(4)] For every \(1\leq\kappa\leq d\),
\[
 \|X_{\kappa,n}\|\longrightarrow2\sigma_\kappa.
\]
\end{enumerate}
Assertions~(1)--(3) also hold when \(Q\) has coefficients in a fixed matrix
algebra.
\end{corollary}

\subsection{Scope of the assumptions}

The nested-corner construction in Definition~\ref{def:iid-entry-model} is a
coupling convention across matrix sizes, not an additional restriction on the
law at any fixed size.  For each \(n\), the entries of \(X_{\kappa,n}\) have
the usual i.i.d. distribution.  The coupling gives a precise meaning to an
almost-sure limit as \(n\to\infty\) and is used when the Borel--Cantelli and
Bai--Yin events are intersected over all sizes.  An almost-sure assertion is
not invariant under an arbitrary recoupling of the sequence, so fixing this
common probability space is part of the formulation rather than a hidden
independence assumption.

The fourth-moment hypothesis enters only in the passage from bounded entries
to the original laws.  The bounded-entry argument of
Section~\ref{sec:bounded} uses boundedness to obtain
\(R(L_n)=O(n^{-1/2})\).  In Section~\ref{sec:tail}, the Bai--Yin bound is
applied to the real and imaginary parts of each fixed centered tail law and
to the original entry law; this is precisely where finite fourth moments are
used.  Thus no higher moment is implicit in the Gaussian comparison or in the
Wick integrability estimate.  The fourth-moment assumption is sharp for a
theorem uniform over the present model class.  Indeed, already in the
one-color real-valued subclass, the choice \(Q(x)=x\) would force
\(\|X_n\|\to2\sigma\); the sample-covariance necessity result of
Bai, Silverstein, and Yin~\cite{bai-silverstein-yin-1988} rules out this
almost-sure edge limit when the fourth moment is infinite.  This sharpness
statement concerns the tuple-level theorem and does not assert that every
individual polynomial requires a fourth moment.

The pseudo-variance is retained throughout the finite-dimensional comparison.
Indeed, the covariance matrix in \eqref{eq:real-covariance} matches it exactly.
The identity \eqref{eq:exact-variance} shows that its contribution to the
variance sum of a fixed pencil is of order \(n^{-1}\), while
Proposition~\ref{prop:gaussian-limit} identifies the limiting Gaussian
coordinates as circular.  Consequently, real entry laws and complex entry
laws without rotational symmetry are covered without first replacing them by
circularly symmetric variables.

Finally, the matrix-coefficient conclusion in Theorem~\ref{thm:main} is not
merely notational.  Self-adjoint linearization tests a polynomial through
pencils of varying fixed sizes, and the same amplified convergence yields the
spectral and inverse-stability conclusions in
Corollary~\ref{cor:spectral-consequences}.

\section{Comparison and linearization}\label{sec:inputs}

Let us define
\[
 H=H_0+\sum_{\alpha\in I}Z_\alpha\in M_D(\mathbb C)_{\rm sa},
\]
where \(H_0\) is deterministic and the \(Z_\alpha\) are independent,
centered, self-adjoint random matrices.  We use the parameters
\begin{align*}
 R(H)&=\max_{\alpha\in I}\|Z_\alpha\|_{L^\infty},\\
 \sigma(H)^2&=\left\|\sum_{\alpha\in I}\E Z_\alpha^2\right\|,\\
 \sigma_*(H)^2
 &=\sup_{\|u\|=\|v\|=1}
   \sum_{\alpha\in I}\E|\langle u,Z_\alpha v\rangle|^2.
\end{align*}
The last equality follows from independence and centering and agrees with the
weak variance parameter in~\cite{brailovskaya-van-handel-2024}.  Let \(G\) be
the self-adjoint Gaussian matrix with the same mean and real entry covariance
as \(H\).

\begin{theorem}[Brailovskaya--van Handel comparison estimates]
\label{thm:bvh-inputs}
With \(H\) and \(G\) as above, the following statements hold.  For
\(t\geq0\),
\begin{equation}\label{eq:bvh-spectrum}
 \mathbb P\left\{
 d_{\rm H}(\spn(H),\spn(G))>C\varepsilon_H(t)
 \right\}\leq D e^{-t},
\end{equation}
where
\begin{equation}\label{eq:bvh-spectrum-error}
 \varepsilon_H(t)
 =\sigma_*(H)t^{1/2}
  +R(H)^{1/3}\sigma(H)^{2/3}t^{2/3}
  +R(H)t.
\end{equation}
Here \(d_{\rm H}\) is Hausdorff distance.  For every fixed \(r\geq1\),
Theorem~2.9 with its parameter \(q=\infty\) gives
\begin{equation}\label{eq:bvh-moment}
 \left|
  \bigl(\E\trn_D H^{2r}\bigr)^{1/(2r)}
  -\bigl(\E\trn_D G^{2r}\bigr)^{1/(2r)}
 \right|
 \leq C_r\bigl(R(H)^{1/3}\sigma(H)^{2/3}+R(H)\bigr).
\end{equation}
Finally, for \(t\geq r\),
\begin{align}
& \mathbb P\bigg\{
 \left|(\trn_D H^{2r})^{ \frac{1}{2r}}
       -\bigl(\E[\trn_D H^{2r}]\bigr)^{ \frac{1}{2r}}\right|
 \notag 
 >C\left(\sigma_*(H)
       +R(H)^{1/2}\bigl(\E[\trn_DH^{2r}]\bigr)^{ \frac{1}{4r}}\right)t^{\frac{1}{2}}
       +CR(H)t\bigg\} \\
 & \leq 2e^{-t}.
\label{eq:bvh-concentration}
\end{align}
The constants in \eqref{eq:bvh-spectrum} and
\eqref{eq:bvh-concentration} are universal, while \(C_r\) may depend on \(r\).
These are Theorems~2.6 and~2.9, with \(q=\infty\), and Lemma~9.20
of~\cite{brailovskaya-van-handel-2024}, written with normalized traces.
\end{theorem}

We also use the following form of the linearization criterion.  Its two
hypotheses are precisely those of Proposition~9.18 of
\cite{brailovskaya-van-handel-2024}; the reverse norm inequality is included
below.

\begin{proposition}[Self-adjoint pencil criterion]\label{prop:pencil-criterion}
Let \(H_{1,n},\ldots,H_{m,n}\) be self-adjoint random matrices, and let
\(s_1,\ldots,s_m\) be a semicircular family in a tracial
\(C^*\)-probability space with faithful trace.  Suppose that, on one
probability-one event, the following statements hold for every \(q\geq1\) and
all self-adjoint \(A_0,\ldots,A_m\in M_q(\mathbb C)\):
\begin{enumerate}
\item[(1)] for every \(\epsilon>0\),
\[
 \spn\!\left(A_0\otimes I_n+\sum_{j=1}^mA_j\otimes H_{j,n}\right)
 \subseteq
 \spn\!\left(A_0\otimes\one+\sum_{j=1}^mA_j\otimes s_j\right)
 +[-\epsilon,\epsilon]
\]
for all sufficiently large \(n\);
\item[(2)] for every \(r\geq1\),
\[
 \trn_{qn}\!\left(A_0\otimes I_n+
       \sum_{j=1}^mA_j\otimes H_{j,n}\right)^{2r}
 \longrightarrow
 (\trn_q\otimes\tau)\!\left(A_0\otimes\one+
       \sum_{j=1}^mA_j\otimes s_j\right)^{2r}.
\]
\end{enumerate}
Then \((H_{1,n},\ldots,H_{m,n})\) converges almost surely strongly in
\(*\)-distribution to \((s_1,\ldots,s_m)\).  On the same event, norm and
normalized-trace convergence also hold for every polynomial with fixed matrix
coefficients.
\end{proposition}

The formulation above also covers semicircular coordinates with unequal
variances: rescale every nonzero coordinate and absorb its standard deviation
into the corresponding coefficient matrix; coordinates of variance zero may
simply be omitted.

\begin{proof}
Proposition~9.18(a) of~\cite{brailovskaya-van-handel-2024} gives, for every
noncommutative polynomial \(Q\),
\[
 \limsup_{n\to\infty}\|Q(H_{1,n},\ldots,H_{m,n})\|
 \leq\|Q(s_1,\ldots,s_m)\|.
\]
Proposition~9.18(b) gives convergence of all joint normalized traces.  Apply
the latter conclusion to \((Q^*Q)^r\).  For every fixed \(r\geq1\),
\begin{align*}
 \liminf_{n\to\infty}\|Q(H_{1,n},\ldots,H_{m,n})\|
 &\geq \lim_{n\to\infty}
 \left\{\trn_n\!\left[(Q^*Q)^r
       (H_{1,n},\ldots,H_{m,n})\right]\right\}^{1/(2r)}\\
 &=\tau\!\left[(Q^*Q)^r(s_1,\ldots,s_m)\right]^{1/(2r)}.
\end{align*}
If \(a=Q(s_1,\ldots,s_m)^*Q(s_1,\ldots,s_m)\), faithfulness of \(\tau\)
implies
\[
 \lim_{r\to\infty}\tau(a^r)^{1/r}=\|a\|.
\]
Indeed, the spectral measure of \(a\) has support equal to \(\spn(a)\), so its
\(L^r\)-norms increase to the endpoint of that support.  Letting
\(r\to\infty\) proves the reverse norm inequality.

For a polynomial with coefficients in \(M_k(\mathbb C)\), use the same
linearization argument with coefficient matrices in \(M_{kq}(\mathbb C)\).
The hypotheses are available for every matrix size.  Proposition~9.18 then
gives the matrix-valued upper norm bound and normalized-trace convergence.
Repeating the preceding lower-bound argument with the faithful trace
\(\trn_k\otimes\tau\) gives the reverse matrix-valued norm inequality.
\end{proof}

\begin{theorem}[Bai--Yin norm bound]\label{thm:bai-yin-input}
If
\(U_n=n^{-1/2}(u_{ij})_{1\leq i,j\leq n}\) is formed from the upper-left
corners of one infinite i.i.d. real array with mean zero, variance \(v^2\),
and finite fourth moment, then Theorem~2.1 of Bai and Yin
\cite{bai-yin}, with one matrix factor, gives
\begin{equation}\label{eq:bai-yin-input}
 \limsup_{n\to\infty}\|U_n\|\leq2v
 \qquad\text{almost surely}.
\end{equation}
\end{theorem}

\section{The bounded-entry case}\label{sec:bounded}

Throughout this section, assume in addition that every entry law
\(\xi^{(\kappa)}\) is essentially bounded.  Set
\[
 S_{\kappa,n}=\frac{X_{\kappa,n}+X_{\kappa,n}^*}{\sqrt2},
 \qquad
 T_{\kappa,n}=\frac{X_{\kappa,n}-X_{\kappa,n}^*}{i\sqrt2}.
\]
For \(q\geq1\) and self-adjoint matrices
\(A_0,A_{\kappa,1},A_{\kappa,2}\in M_q(\mathbb C)\), define
\begin{equation}\label{eq:pencil-Ln}
 L_n=A_0\otimes I_n+
 \sum_{\kappa=1}^d
 \left(A_{\kappa,1}\otimes S_{\kappa,n}
      +A_{\kappa,2}\otimes T_{\kappa,n}\right).
\end{equation}
Let
\[
 B_\kappa=\frac{A_{\kappa,1}-iA_{\kappa,2}}{\sqrt2}.
\]
Then
\begin{equation}\label{eq:pencil-complex-form}
 L_n=A_0\otimes I_n+
 \sum_{\kappa=1}^d
 \left(B_\kappa\otimes X_{\kappa,n}
      +B_\kappa^*\otimes X_{\kappa,n}^*\right).
\end{equation}

\subsection{The pencil parameters}

\begin{lemma}[Entry-summand estimates]\label{lem:pencil-parameters}
For every fixed pencil \eqref{eq:pencil-Ln}, regarded as a
\(D_n\times D_n\) matrix with \(D_n=qn\),
\[
 R(L_n)=O(n^{-1/2}),
 \qquad
 \sigma(L_n)=O(1),
 \qquad
 \sigma_*(L_n)=O(n^{-1/2}).
\]
The constants may depend on the entry laws and the fixed pencil, but not on
\(n\).
\end{lemma}

\begin{proof}
Write \(E_{ij}\) for the matrix units in \(M_n(\mathbb C)\).  The random part
of \(L_n\) is the sum of the independent self-adjoint matrices
\begin{equation}\label{eq:entry-summand}
 Z_{\kappa ij}^{(n)}
 =\frac1{\sqrt n}\left(
 B_\kappa\otimes\xi^{(\kappa)}_{ij}E_{ij}
 +B_\kappa^*\otimes\overline{\xi^{(\kappa)}_{ij}}E_{ji}
 \right).
\end{equation}
If \(b=\max_\kappa\|\xi^{(\kappa)}\|_{L^\infty}\), then
\[
 \|Z_{\kappa ij}^{(n)}\|
 \leq\frac{2b\|B_\kappa\|}{\sqrt n},
\]
which proves the estimate for \(R(L_n)\).

Put \(m_\kappa=\E[(\xi^{(\kappa)})^2]\).  Multiplication of the two matrix
units in \eqref{eq:entry-summand} gives the exact identity
\begin{align}
 \sum_{\kappa,i,j}\E[(Z_{\kappa ij}^{(n)})^2]
 =\sum_{\kappa=1}^d\bigg[&
 \sigma_\kappa^2(B_\kappa B_\kappa^*+B_\kappa^*B_\kappa)
 \notag\\
 &+\frac1n\bigl(m_\kappa B_\kappa^2
       +\overline{m_\kappa}(B_\kappa^*)^2\bigr)\bigg]\otimes I_n.
\label{eq:exact-variance}
\end{align}
The right-hand side has bounded norm, so \(\sigma(L_n)=O(1)\).

For the weak variance estimate, decompose unit vectors in
\(\mathbb C^q\otimes\mathbb C^n\) as \(u=(u_i)_{i=1}^n\) and
\(v=(v_j)_{j=1}^n\), where \(u_i,v_j\in\mathbb C^q\).  We have
\[
 \sum_{i,j}|\langle u_i,B_\kappa v_j\rangle|^2
 \leq\|B_\kappa\|^2
 \left(\sum_i\|u_i\|^2\right)
 \left(\sum_j\|v_j\|^2\right)
 =\|B_\kappa\|^2,
\]
and the same estimate holds with \(B_\kappa^*\).  Since
\[
 \E|\xi a+\overline\xi b|^2
 \leq2\E|\xi|^2(|a|^2+|b|^2),
\]
independence and centering give
\[
 \sigma_*(L_n)^2
 \leq\frac4n\sum_{\kappa=1}^d
 \sigma_\kappa^2\|B_\kappa\|^2.
\]
This completes the proof.
\end{proof}

\subsection{The matching Gaussian pencil}

\begin{lemma}[Gaussian realization]\label{lem:gaussian-realization}
For each \(\kappa\), there is a centered complex Gaussian variable
\(g^{(\kappa)}\) such that
\[
 \E|g^{(\kappa)}|^2=\sigma_\kappa^2,
 \qquad
 \E[(g^{(\kappa)})^2]=m_\kappa.
\]
If \(Y_{\kappa,n}=n^{-1/2}(g^{(\kappa)}_{ij})\) are independent i.i.d.
Gaussian matrices, then the Gaussian comparison matrix for \(L_n\) has the
same law as
\begin{equation}\label{eq:Gaussian-pencil}
 G_n=A_0\otimes I_n+
 \sum_{\kappa=1}^d
 \left(B_\kappa\otimes Y_{\kappa,n}
      +B_\kappa^*\otimes Y_{\kappa,n}^*\right).
\end{equation}
\end{lemma}

\begin{proof}
The covariance matrix required for the real and imaginary parts of
\(g^{(\kappa)}\) is
\begin{equation}\label{eq:real-covariance}
 \frac12
 \begin{pmatrix}
  \sigma_\kappa^2+\operatorname{Re}m_\kappa&
  \operatorname{Im}m_\kappa\\
  \operatorname{Im}m_\kappa&
  \sigma_\kappa^2-\operatorname{Re}m_\kappa
 \end{pmatrix}.
\end{equation}
It is positive semidefinite: its trace is \(\sigma_\kappa^2\), and its
determinant is
\(\tfrac14(\sigma_\kappa^4-|m_\kappa|^2)\geq0\), because
\(|m_\kappa|\leq\E|\xi^{(\kappa)}|^2=\sigma_\kappa^2\).
Thus the required Gaussian variable exists, including when the covariance is
degenerate.

Every covariance of the summands \eqref{eq:entry-summand} is a linear
combination of \(\E|\xi^{(\kappa)}|^2\),
\(\E[(\xi^{(\kappa)})^2]\), and its complex conjugate.  Replacing each entry
by \(g^{(\kappa)}\) therefore preserves the full real covariance of the
self-adjoint matrix \(L_n\).  Independence of entries and colors is also
preserved, so the uniqueness in law of a Gaussian vector proves
\eqref{eq:Gaussian-pencil}.
\end{proof}

\begin{theorem}[Anderson's strong Wigner theorem]
\label{thm:anderson-input}
For each \(1\leq j\leq m\), let \(W_{j,n}\) be \(n^{-1/2}\) times the
upper-left \(n\times n\) corner of an infinite Hermitian Wigner array
satisfying the hypotheses of
\cite[Theorems~1 and~2 and Corollary~1]{anderson-2013-polynomial}: its
off-diagonal entries are centered with unit second absolute moment and finite
fourth moment, its diagonal entries are centered real variables with finite
fourth moment, and the real and imaginary parts of each off-diagonal entry
are independent.  Assume that the arrays are independent.  Then
\((W_{1,n},\ldots,W_{m,n})\) converges almost surely strongly in
\(*\)-distribution to a free standard semicircular family.  The conclusion
also holds for polynomials with fixed matrix coefficients.
\end{theorem}

\begin{lemma}[Spectra under strong convergence]\label{lem:strong-spectrum}
Let \(A_n=A_n^*\) converge strongly in distribution to a self-adjoint element
\(a\) in a tracial \(C^*\)-probability space with faithful trace.  Then
\[
 d_{\rm H}(\spn(A_n),\spn(a))\longrightarrow0.
\]
\end{lemma}

\begin{proof}
Norm convergence of the coordinate polynomial makes the spectra uniformly
bounded, say in a fixed interval \([-M,M]\).  Fix \(\epsilon>0\).

For the upper inclusion, choose a continuous function \(f:[-M,M]\to[0,1]\)
that vanishes on \(\spn(a)\) and equals one at every point whose distance from
\(\spn(a)\) is at least \(\epsilon\).  Approximate \(f\) uniformly by a real
polynomial \(p\) with error less than \(1/4\).  Then
\(\|p(a)\|<1/4\), and strong convergence gives \(\|p(A_n)\|<1/2\) for all
large \(n\).  If \(\spn(A_n)\) contained a point at distance at least
\(\epsilon\) from \(\spn(a)\), the spectral mapping theorem would instead
give \(\|p(A_n)\|\geq3/4\), a contradiction.

For the lower inclusion, fix \(\lambda\in\spn(a)\) and choose a nonnegative
continuous function \(h\) supported in
\((\lambda-\epsilon,\lambda+\epsilon)\) with \(h(\lambda)>0\).  Functional
calculus gives \(h(a)\neq0\); faithfulness gives \(\tau(h(a))>0\).  Uniform
polynomial approximation, moment convergence, and the uniform spectral bound
imply
\(\trn_n h(A_n)\to\tau(h(a))\).  Hence \(h(A_n)\neq0\) for all large \(n\),
so \(\spn(A_n)\) meets
\((\lambda-\epsilon,\lambda+\epsilon)\).  Compactness of \(\spn(a)\)
reduces this to finitely many \(\lambda\)'s and completes the proof.
\end{proof}

\begin{proposition}[Gaussian limit]\label{prop:gaussian-limit}
The Gaussian matrices from Lemma~\ref{lem:gaussian-realization} converge
almost surely strongly in \(*\)-distribution.  Their limit
\((z_1,\ldots,z_d)\) is a free circular family satisfying
\(\tau(z_\kappa z_\kappa^*)=\sigma_\kappa^2\).
Consequently, for every fixed self-adjoint pencil,
\begin{equation}\label{eq:Gaussian-spectrum-limit}
 d_{\rm H}(\spn(G_n),\spn(G_\infty))\longrightarrow0
 \qquad\text{almost surely},
\end{equation}
where
\[
 G_\infty=A_0\otimes\one+
 \sum_{\kappa=1}^d
 (B_\kappa\otimes z_\kappa+B_\kappa^*\otimes z_\kappa^*).
\]
\end{proposition}

\begin{proof}
Take independent standard real Gaussian variables \(r^{(\kappa)}\) and
\(s^{(\kappa)}\).  A square root of \eqref{eq:real-covariance} yields complex
numbers \(a_\kappa,b_\kappa\) such that
\[
 g^{(\kappa)}=a_\kappa r^{(\kappa)}+b_\kappa s^{(\kappa)},
 \qquad
 |a_\kappa|^2+|b_\kappa|^2=\sigma_\kappa^2,
 \qquad
 a_\kappa^2+b_\kappa^2=m_\kappa.
\]
Accordingly,
\[
 Y_{\kappa,n}=a_\kappa R_{\kappa,n}+b_\kappa Q_{\kappa,n},
\]
where all \(R_{\kappa,n},Q_{\kappa,n}\) are independent real Ginibre matrices
with entries \(N(0,n^{-1})\).

For one real Ginibre matrix \(R_n\), set
\[
 H_n=\frac{R_n+R_n^*}{\sqrt2},
 \qquad
 K_n=\frac{R_n-R_n^*}{i\sqrt2}.
\]
For \(i<j\), the pair
\[
 \frac{R_n(i,j)+R_n(j,i)}{\sqrt2},
 \qquad
 \frac{R_n(i,j)-R_n(j,i)}{\sqrt2}
\]
is an orthogonal transform of two independent real Gaussians.  Thus \(H_n\)
and \(K_n\) are independent Wigner matrices.  Their off-diagonal entries have
variance \(n^{-1}\); the diagonal variances are \(2n^{-1}\) and \(0\),
respectively.  Purely real, purely imaginary, and zero components are allowed
by Anderson's assumptions.  Since
\[
 R_n=\frac{H_n+iK_n}{\sqrt2},
\]
Theorem~\ref{thm:anderson-input}, applied jointly to all the Wigner matrices
just constructed, gives strong convergence to a free semicircular family.
Consequently, the Ginibre matrices converge strongly to a free circular
family.  Denote the two circular limits associated with
\(R_{\kappa,n}\) and \(Q_{\kappa,n}\) by
\(c_{\kappa,1}\) and \(c_{\kappa,2}\), each of variance one.  Then
\[
 z_\kappa=a_\kappa c_{\kappa,1}+b_\kappa c_{\kappa,2}.
\]

The circular variables are free Gaussian in the free-probability sense.
Multilinearity of free cumulants and freeness give
\begin{align*}
 \kappa_2(z_\kappa,z_\kappa)
 &=\kappa_2(z_\kappa^*,z_\kappa^*)=0,\\
 \kappa_2(z_\kappa,z_\kappa^*)
 &=\kappa_2(z_\kappa^*,z_\kappa)
 =|a_\kappa|^2+|b_\kappa|^2=\sigma_\kappa^2,
\end{align*}
and all cumulants of order other than two vanish.  Hence \(z_\kappa\) is
circular with the required variance.  Different colors are free because they
belong to disjoint free subfamilies.  Finally,
Lemma~\ref{lem:strong-spectrum}, applied to the matrix-valued pencil, proves
\eqref{eq:Gaussian-spectrum-limit}.
\end{proof}

\subsection{Gaussian trace integrability}

The moment comparison \eqref{eq:bvh-moment} involves expectations.  We next
justify passage from Anderson's almost-sure Gaussian trace limit to the
corresponding expected limit.

\begin{lemma}[Wick index count for products of traces]
\label{lem:Wick-index-count}
Fix \(h\geq1\) and \(p_1,\ldots,p_h\geq0\).  Expand
\[
 \prod_{a=1}^h\trn_{qn}G_n^{p_a}
\]
using the independent real Ginibre representation from the proof of
Proposition~\ref{prop:gaussian-limit}.  Consider a monomial containing
\(2e\) centered Gaussian matrix entries and a Wick pairing whose covariance
factors are nonzero.  After all covariance identifications, let \(v\) be the
number of free \(n\)-indices.  Then
\begin{equation}\label{eq:Wick-index-count}
 v\leq e+h.
\end{equation}
Consequently, the contribution of every fixed nonzero pairing to the product
of normalized traces is \(O(1)\), uniformly in \(n\).
\end{lemma}

\begin{proof}
Write out the block indices in \(\{1,\ldots,q\}\) and the matrix indices
in \(\{1,\ldots,n\}\).  The block-index sums have a fixed finite size, so only
the latter indices affect the power of \(n\).  In each cyclic trace, a factor
\(A_0\otimes I_n\) equates its two adjacent \(n\)-indices; contract every such
deterministic stretch.  What remains from each trace is either one isolated
vertex, if the monomial contains no random factor in that trace, or a
connected cyclic multigraph whose edges are the Gaussian occurrences.

A nonzero covariance pairs two occurrences from the same underlying real
Gaussian array.  It identifies their two ordered endpoints, with the
orientation determined by whether the corresponding matrix occurrence was
transposed.  Taking the quotient by all these identifications can merge
components but cannot create new ones.  The quotient graph therefore has
\(e\) random edge classes and at most \(h\) connected components.  If its
components have \(v_j\) vertices and \(e_j\) edge classes, then
\(v_j\leq e_j+1\); this also covers an isolated component, for which
\(e_j=0\) and \(v_j=1\).  Hence
\[
 v=\sum_jv_j\leq\sum_j(e_j+1)=e+\#\{\text{components}\}\leq e+h,
\]
which proves \eqref{eq:Wick-index-count}.

Each covariance is \(O(n^{-1})\), so the \(e\) pairs contribute
\(O(n^{-e})\).  The index summation contributes \(n^v\), while the \(h\)
normalized traces contribute \(n^{-h}\) up to the fixed factor \(q^{-h}\).
Thus the total power is \(n^{v-e-h}\leq1\).  All remaining coefficient and
block-index factors are independent of \(n\).
\end{proof}

\begin{lemma}[Uniform \(L^2\) bound for Gaussian pencil traces]
\label{lem:Gaussian-L2}
For every fixed self-adjoint Gaussian pencil \eqref{eq:Gaussian-pencil} and
every fixed \(r\geq1\),
\[
 \sup_{n\geq1}
 \E\left|\trn_{qn}G_n^{2r}\right|^2<\infty.
\]
Consequently,
\begin{equation}\label{eq:Gaussian-expected-moment}
 \E\trn_{qn}G_n^{2r}
 \longrightarrow
 (\trn_q\otimes\tau)(G_\infty^{2r}).
\end{equation}
\end{lemma}

\begin{proof}
The matrix \(G_n\) is self-adjoint, so
\(\trn_{qn}G_n^{2r}\) is nonnegative.  Use the representation in the proof of
Proposition~\ref{prop:gaussian-limit} and expand
\(\E|\trn_{qn}G_n^{2r}|^2\) as a product of two normalized traces.  There are
finitely many monomials, depending only on \(q,d,r\) and the fixed
coefficients.  A monomial with an odd number of centered Gaussian entries has
zero expectation.  For a monomial with \(2e\) such entries, Wick's formula
gives finitely many pairings.  Lemma~\ref{lem:Wick-index-count}, with \(h=2\),
shows that every nonzero paired contribution is \(O(1)\), uniformly in
\(n\).  Summing the fixed finite collection of monomials and pairings proves
the uniform \(L^2\) bound.

Anderson's trace convergence gives
\[
 \trn_{qn}G_n^{2r}
 \longrightarrow(\trn_q\otimes\tau)(G_\infty^{2r})
 \qquad\text{almost surely}.
\]
The \(L^2\) bound implies uniform integrability, so taking expectations yields
\eqref{eq:Gaussian-expected-moment}.
\end{proof}

\subsection{All pencil limits on one event}

\begin{proposition}[Bounded-entry pencil limits]\label{prop:bounded-pencil-limits}
There is a probability-one event on which, simultaneously for every
\(q\geq1\), all self-adjoint coefficient matrices in
\eqref{eq:pencil-Ln}, and every \(r\geq1\),
\begin{align}
 d_{\rm H}(\spn(L_n),\spn(G_\infty))&\longrightarrow0,
 \label{eq:pencil-spectrum-final}\\
 \trn_{qn}L_n^{2r}
 &\longrightarrow(\trn_q\otimes\tau)(G_\infty^{2r}).
 \label{eq:pencil-moment-final}
\end{align}
\end{proposition}

\begin{proof}
We begin with a fixed \(q\) and a pencil whose coefficients have entries in
\(\mathbb Q+i\mathbb Q\), and put \(D_n=qn\).

\smallskip
\noindent\emph{Spectrum for rational pencils.}
Apply \eqref{eq:bvh-spectrum} with
\(t=4\log D_n\).  Lemma~\ref{lem:pencil-parameters} gives
\begin{align*}
 \varepsilon_{L_n}(4\log D_n)
 &=O\left(
 n^{-1/2}(\log n)^{1/2}
 +n^{-1/6}(\log n)^{2/3}
 +n^{-1/2}\log n\right)\\
 &=o(1),
\end{align*}
and the exceptional probability is at most \(D_n^{-3}\).  Enlarge the
original probability space by independent infinite Gaussian arrays and form
all \(Y_{\kappa,n}\) as nested upper-left corners.  The spectrum comparison
estimate is valid for this product coupling; in fact Theorem~2.6 of
\cite{brailovskaya-van-handel-2024} is valid for every coupling of the two
models.  For the fixed rational pencil, Borel--Cantelli and
\eqref{eq:Gaussian-spectrum-limit} yield
\eqref{eq:pencil-spectrum-final} on the product space.

\smallskip
\noindent\emph{Moments for rational pencils.}
Set
\[
 a_n=(\E\trn_{D_n}L_n^{2r})^{1/(2r)},
 \qquad
 b_n=(\E\trn_{D_n}G_n^{2r})^{1/(2r)}.
\]
By Lemma~\ref{lem:Gaussian-L2},
\[
 b_n\longrightarrow
 b_\infty:=(\trn_q\otimes\tau)(G_\infty^{2r})^{1/(2r)}.
\]
The moment comparison \eqref{eq:bvh-moment} and
Lemma~\ref{lem:pencil-parameters} give
\[
 |a_n-b_n|
 \leq C_r\bigl(R(L_n)^{1/3}\sigma(L_n)^{2/3}+R(L_n)\bigr)
 =o(1).
\]
Thus \(a_n\to b_\infty\), in particular \(\sup_n a_n<\infty\) after
discarding finitely many \(n\).

Use \eqref{eq:bvh-concentration} with \(t=4\log D_n\), which is at least \(r\)
for all large \(n\).  The deviation threshold is
\begin{align*}
 &C\left(\sigma_*(L_n)+R(L_n)^{1/2}a_n^{1/2}\right)
 (4\log D_n)^{1/2}+4CR(L_n)\log D_n\\
 &\hspace{25mm}=O\left(n^{-1/4}(\log n)^{1/2}
                   +n^{-1/2}\log n\right)=o(1),
\end{align*}
and the exceptional probability is at most \(2D_n^{-4}\).  Another
Borel--Cantelli argument gives
\[
 (\trn_{D_n}L_n^{2r})^{1/(2r)}\longrightarrow b_\infty
 \qquad\text{almost surely}.
\]
Raising to the power \(2r\) proves \eqref{eq:pencil-moment-final}.

\smallskip
\noindent\emph{One event and removal of rationality.}
There are only countably many rational pencils, values of \(q\), and values
of \(r\).  On the product space, intersect the spectrum events above over all
rational pencils with the one event on which the Gaussian family converges
strongly.  The resulting event still has probability one.  Its conclusion
\eqref{eq:pencil-spectrum-final} refers only to the original matrices and the
deterministic free limit, not to the sampled Gaussian arrays.  Let
\(\mathcal E_{\rm sp}\) denote the event in the original probability space
on which all these conclusions hold.  The product-space event just
constructed is contained in
\(\mathcal E_{\rm sp}\times\Omega_{\rm G}\), where \(\Omega_{\rm G}\)
denotes the Gaussian factor; hence Fubini's theorem gives
\(\mathbb P(\mathcal E_{\rm sp})=1\).  Independently, intersect the moment
events over all rational pencils and all \(r\).  Combining the two gives one
event in the original space on which both conclusions hold for every rational
pencil.

The coordinate pencils show on this event that every sequence
\(\|S_{\kappa,n}\|\) and \(\|T_{\kappa,n}\|\) is bounded.  We now approximate
arbitrary self-adjoint coefficient matrices by rational self-adjoint ones.
For two coefficient lists \(A,A'\),
\begin{align*}
 \|L_n(A)-L_n(A')\|
 \leq&\ \|A_0-A_0'\|\\
 &+\sum_{\kappa=1}^d
 \left(\|A_{\kappa,1}-A_{\kappa,1}'\|\,\|S_{\kappa,n}\|
      +\|A_{\kappa,2}-A_{\kappa,2}'\|\,\|T_{\kappa,n}\|\right).
\end{align*}
The analogous estimate holds for the free pencil.  Hausdorff spectral
distance between self-adjoint elements is bounded by their norm distance, so
\eqref{eq:pencil-spectrum-final} extends to arbitrary coefficients.  Finally,
for self-adjoint \(C,D\),
\[
 |\trn(C^{2r})-\trn(D^{2r})|
 \leq2r\max(\|C\|,\|D\|)^{2r-1}\|C-D\|.
\]
This extends \eqref{eq:pencil-moment-final} and completes the proof.
\end{proof}

\begin{theorem}[Bounded-entry strong circular limit]
\label{thm:bounded-strong-limit}
Under Definition~\ref{def:iid-entry-model}, if all entry laws are essentially
bounded, then \((X_{1,n},\ldots,X_{d,n})\) converges almost surely strongly in
\(*\)-distribution to the free circular family with variances
\(\sigma_1^2,\ldots,\sigma_d^2\).  The convergence holds on the same event
for polynomials with fixed matrix coefficients.
\end{theorem}

\begin{proof}
Let
\[
 s_{\kappa,1}=\frac{z_\kappa+z_\kappa^*}{\sqrt2},
 \qquad
 s_{\kappa,2}=\frac{z_\kappa-z_\kappa^*}{i\sqrt2}.
\]
By Proposition~\ref{prop:gaussian-limit}, these \(2d\) elements form a
semicircular family, with variance \(\sigma_\kappa^2\) in each coordinate.
Proposition~\ref{prop:bounded-pencil-limits} verifies both hypotheses of
Proposition~\ref{prop:pencil-criterion} for the self-adjoint tuple
\((S_{1,n},T_{1,n},\ldots,S_{d,n},T_{d,n})\).  Hence that tuple converges
strongly to
\((s_{1,1},s_{1,2},\ldots,s_{d,1},s_{d,2})\).  Since
\(X_{\kappa,n}=(S_{\kappa,n}+iT_{\kappa,n})/\sqrt2\), the asserted circular
strong convergence follows.
\end{proof}

\section{Removal of the boundedness assumption}\label{sec:tail}

\begin{lemma}[Polynomial continuity on bounded sets]
\label{lem:polynomial-continuity}
Let \(Q=\sum_w\alpha_ww\) be a noncommutative \(*\)-polynomial in
\(d\) variables.  For \(M\geq1\), set
\[
 \operatorname{Lip}_M(Q)
 =\sum_{|w|\geq1}|\alpha_w|\,|w|M^{|w|-1}.
\]
If two tuples \(A=(A_1,\ldots,A_d)\) and \(B=(B_1,\ldots,B_d)\) satisfy
\[
 \max_\kappa\bigl(\|A_\kappa\|,\|B_\kappa\|\bigr)\leq M,
 \qquad
 \delta=\max_\kappa\|A_\kappa-B_\kappa\|,
\]
then
\begin{equation}\label{eq:polynomial-continuity}
 \|Q(A,A^*)-Q(B,B^*)\|\leq\operatorname{Lip}_M(Q)\delta.
\end{equation}
The same right-hand side bounds the difference of the normalized traces when
\(A\) and \(B\) are matrix tuples of the same size.
More generally, if \(Q=\sum_w C_w\otimes w\) has coefficients in a fixed
matrix algebra, the conclusions hold with
\[
 \operatorname{Lip}_M(Q)
 =\sum_{|w|\geq1}\|C_w\|\,|w|M^{|w|-1}
\]
and with the tensor-product operator norm and normalized trace.
\end{lemma}

\begin{proof}
Write a word of length \(\ell\) as \(w=x_1\cdots x_\ell\), where every
letter is one of the variables or its adjoint.  If \(a_j\) and \(b_j\) are
the corresponding evaluations, then
\[
 a_1\cdots a_\ell-b_1\cdots b_\ell
 =\sum_{j=1}^\ell
   a_1\cdots a_{j-1}(a_j-b_j)b_{j+1}\cdots b_\ell.
\]
Taking norms gives
\(\|w(A,A^*)-w(B,B^*)\|\leq\ell M^{\ell-1}\delta\), because
\(\|A_\kappa^*-B_\kappa^*\|=\|A_\kappa-B_\kappa\|\).  Summing over the
words proves \eqref{eq:polynomial-continuity}.  The trace assertion follows
from \(|\trn_n C|\leq\|C\|\).  For matrix coefficients, apply the same word
identity after tensoring with \(C_w\), take norms, and sum over \(w\).
\end{proof}

For every integer \(K\geq1\), define the centered truncation
\begin{equation}\label{eq:fixed-truncation}
 \xi^{(\kappa,K)}
 =\xi^{(\kappa)}\mathbf1_{\{|\xi^{(\kappa)}|\leq K\}}
 -\E\!\left[\xi^{(\kappa)}
       \mathbf1_{\{|\xi^{(\kappa)}|\leq K\}}\right]
\end{equation}
and let \(X_{\kappa,n}^{(K)}\) be the matrix formed from the corresponding
nested i.i.d. array.

\begin{lemma}[Fixed-level finite-fourth tail removal]
\label{lem:fixed-tail}
Under Definition~\ref{def:iid-entry-model}, almost surely,
\begin{equation}\label{eq:tail-removal}
 \lim_{K\to\infty}\limsup_{n\to\infty}
 \max_{1\leq\kappa\leq d}
 \|X_{\kappa,n}-X_{\kappa,n}^{(K)}\|=0.
\end{equation}
Moreover,
\begin{equation}\label{eq:original-bounded}
 \max_{1\leq\kappa\leq d}\limsup_{n\to\infty}
 \|X_{\kappa,n}\|<\infty
 \qquad\text{almost surely}.
\end{equation}
\end{lemma}

\begin{proof}
Write
\begin{align*}
 \eta^{(\kappa,K)}
 &=\xi^{(\kappa)}-\xi^{(\kappa,K)}\\
 &=\xi^{(\kappa)}\mathbf1_{\{|\xi^{(\kappa)}|>K\}}
  -\E\!\left[\xi^{(\kappa)}
       \mathbf1_{\{|\xi^{(\kappa)}|>K\}}\right]
\end{align*}
and put
\(\rho_{\kappa,K}^2=\E|\eta^{(\kappa,K)}|^2\).  Centering is an orthogonal
projection in \(L^2\), so
\[
 \rho_{\kappa,K}^2
 \leq\E\left[|\xi^{(\kappa)}|^2
       \mathbf1_{\{|\xi^{(\kappa)}|>K\}}\right]
 \longrightarrow0.
\]
Write \(\eta^{(\kappa,K)}=u^{(\kappa,K)}+iv^{(\kappa,K)}\).  Both real
components are centered, i.i.d. across matrix entries, and have finite fourth
moment.  Applying \eqref{eq:bai-yin-input} separately to their nested arrays
and using the triangle inequality gives
\begin{align*}
 \limsup_{n\to\infty}
 \|X_{\kappa,n}-X_{\kappa,n}^{(K)}\|
 &\leq2\sqrt{\E|u^{(\kappa,K)}|^2}
      +2\sqrt{\E|v^{(\kappa,K)}|^2}\\
 &\leq2\sqrt2\,\rho_{\kappa,K}
 \qquad\text{almost surely}.
\end{align*}
The colors and integer levels form a countable collection, so the estimate
holds simultaneously for all of them on one probability-one event.  Sending
\(K\to\infty\) proves \eqref{eq:tail-removal}.

For \eqref{eq:original-bounded}, apply \eqref{eq:bai-yin-input} directly to
the real and imaginary parts of each original entry law.  Their fourth moments
are finite, and there are only finitely many colors.
\end{proof}

We now complete the proof of the main theorem.

\begin{proof}[Proof of Theorem~\ref{thm:main}]
Let
\[
 \sigma_{\kappa,K}^2=\E|\xi^{(\kappa,K)}|^2.
\]
By \(L^2\)-convergence in \eqref{eq:fixed-truncation},
\(\sigma_{\kappa,K}\to\sigma_\kappa\).  For each fixed integer \(K\), the
entries \(\xi^{(\kappa,K)}\) are bounded, so
Theorem~\ref{thm:bounded-strong-limit} gives almost-sure strong convergence of
\((X_{1,n}^{(K)},\ldots,X_{d,n}^{(K)})\) to a free circular family
\((c_1^{(K)},\ldots,c_d^{(K)})\) with variances
\(\sigma_{1,K}^2,\ldots,\sigma_{d,K}^2\).  These assertions hold
simultaneously for all \(K\in\mathbb N\) after taking a countable
intersection.  Intersect once more with the probability-one event from
Lemma~\ref{lem:fixed-tail}, including \eqref{eq:original-bounded}, and denote
the resulting event by \(\Omega_0\).  It is fixed before a polynomial is
chosen.  All the estimates below are pointwise on \(\Omega_0\).

Realize all free limits on the space containing \(c_1,\ldots,c_d\) by setting
\[
 c_\kappa^{(K)}
 =\frac{\sigma_{\kappa,K}}{\sigma_\kappa}c_\kappa
 \quad\text{if }\sigma_\kappa>0.
\]
If \(\sigma_\kappa=0\), then \(\xi^{(\kappa)}=0\) almost surely and we set both
elements equal to zero.  Therefore
\begin{equation}\label{eq:free-K-convergence}
 \max_\kappa\|c_\kappa^{(K)}-c_\kappa\|\longrightarrow0.
\end{equation}

Fix a noncommutative \(*\)-polynomial \(Q\).  Constant polynomials require no
argument, so suppose that \(Q\) has positive degree.
By \eqref{eq:original-bounded}, \eqref{eq:tail-removal}, and
\[
 \|X_{\kappa,n}^{(K)}\|
 \leq\|X_{\kappa,n}\|
     +\|X_{\kappa,n}-X_{\kappa,n}^{(K)}\|,
\]
there are \(K_0\) and a finite random \(M>1\) such that, for every
\(K\geq K_0\),
\[
 \limsup_{n\to\infty}\max_\kappa
 \bigl(\|X_{\kappa,n}\|,\|X_{\kappa,n}^{(K)}\|\bigr)<M.
\]
Thus, for each such \(K\), both tuples are bounded by \(M\) for all
sufficiently large \(n\).  The index may depend on \(K\), which is sufficient
because \(n\to\infty\) is taken first.  Applying
Lemma~\ref{lem:polynomial-continuity} and then
Lemma~\ref{lem:fixed-tail} yields
\begin{equation}\label{eq:polynomial-tail}
 \lim_{K\to\infty}\limsup_{n\to\infty}
 \|Q(X_n,X_n^*)-Q(X_n^{(K)},(X_n^{(K)})^*)\|=0.
\end{equation}
The family consisting of all \(c^{(K)}\) and \(c\) is uniformly bounded,
because \(\sigma_{\kappa,K}\to\sigma_\kappa\).  The same lemma and
\eqref{eq:free-K-convergence} therefore give
\begin{equation}\label{eq:free-polynomial-K}
 \|Q(c^{(K)},(c^{(K)})^*)-Q(c,c^*)\|\longrightarrow0.
\end{equation}

For the norm, write
\begin{align*}
 \big|\|Q(X_n,X_n^*)\|-\|Q(c,c^*)\|\big|
 \leq&\ \|Q(X_n,X_n^*)-Q(X_n^{(K)},(X_n^{(K)})^*)\|\\
 &+\big|\|Q(X_n^{(K)},(X_n^{(K)})^*)\|
          -\|Q(c^{(K)},(c^{(K)})^*)\|\big|\\
 &+\|Q(c^{(K)},(c^{(K)})^*)-Q(c,c^*)\|.
\end{align*}
Take \(\limsup_{n\to\infty}\) and then \(K\to\infty\).  The middle term
vanishes for fixed \(K\) by the bounded-entry theorem, while the first and
last terms vanish by \eqref{eq:polynomial-tail} and
\eqref{eq:free-polynomial-K}.  This proves norm convergence.

For normalized traces, use
\[
 |\trn_n(A)-\trn_n(B)|\leq\|A-B\|
\]
in place of the first norm inequality.  Strong convergence at fixed \(K\),
followed by the same two limiting steps, gives
\[
 \trn_nQ(X_n,X_n^*)\longrightarrow\tau(Q(c,c^*)).
\]
If \(Q\) has coefficients in a fixed \(M_q(\mathbb C)\), repeat the same
argument using the matrix-coefficient assertion of
Lemma~\ref{lem:polynomial-continuity} and the matrix-valued conclusion of
Theorem~\ref{thm:bounded-strong-limit}.  The trace estimate becomes
\[
 |(\trn_q\otimes\trn_n)(A-B)|\leq\|A-B\|,
\]
so the same two limiting steps prove the matrix-valued trace convergence as
well.

The event \(\Omega_0\) was chosen independently of \(q\) and \(Q\), and the
argument above applies there to every fixed matrix size and every fixed
\(*\)-polynomial.  Hence all the asserted convergences hold simultaneously on
one probability-one event.
\end{proof}

\begin{proof}[Proof of Corollary~\ref{cor:homogeneous}]
Apply Theorem~\ref{thm:main} to the \(*\)-polynomial \(Q=P\), which contains
no adjoint variables.
\end{proof}

\begin{proof}[Proof of Corollary~\ref{cor:spectral-consequences}]
For the first assertion, set \(A_n=Q(X_n,X_n^*)\) and
\(a=Q(c,c^*)\).  Theorem~\ref{thm:main} implies strong convergence of
the one-element tuple \(A_n\) to \(a\), so
Lemma~\ref{lem:strong-spectrum} applies.

For the second assertion, apply the first one to the self-adjoint polynomial
\(Q^*Q\).  All the resulting spectra lie in one compact subinterval of
\([0,\infty)\), and the square-root map is uniformly continuous there.
The definition of \(\operatorname{sv}\) therefore gives the claimed
Hausdorff convergence.

Put \(a=Q(c,c^*)\) for the third assertion.  If \(a\) is invertible, then
\(\delta:=\min\operatorname{sv}(a)=\|a^{-1}\|^{-1}>0\).  Hausdorff
convergence of the singular-value spectra gives
\[
 \min\operatorname{sv}(Q(X_n,X_n^*))\longrightarrow\delta.
\]
Hence the matrices are eventually invertible, and taking reciprocals proves
the inverse-norm convergence.

For fixed matrix coefficients, Theorem~\ref{thm:main} gives the same strong
convergence after amplification, and \(\trn_q\otimes\tau\) is faithful.
Thus Lemma~\ref{lem:strong-spectrum} and the preceding two arguments apply
without change in the tensor-product algebra.

Finally, Theorem~\ref{thm:main} with \(Q(x)=x_\kappa\) gives
\(\|X_{\kappa,n}\|\to\|c_\kappa\|\).  The positive element
\(c_\kappa^*c_\kappa\) has the free Poisson law with support
\([0,4\sigma_\kappa^2]\), so faithfulness of the trace gives
\(\|c_\kappa\|=2\sigma_\kappa\); see, for example,
\cite{nica-speicher-2006}.
\end{proof}

\section{Discussion}\label{sec:discussion}

Theorem~\ref{thm:main} should be viewed as a joint strong-convergence statement,
rather than as a collection of norm estimates for individual matrices.  It
places the normalized traces and operator norms of every fixed mixed
\(*\)-polynomial on one probability-one event, and it remains valid after
every fixed matrix amplification. Thus, the conclusion controls
interactions among all colors and their adjoints that cannot be recovered
from the coordinate limits
\(\|X_{\kappa,n}\|\to2\sigma_\kappa\) alone.  The results with self-adjoint spectrum,
singular-value, and inverse-stability in
Corollary~\ref{cor:spectral-consequences} are manifestations of this joint
control.

The methodological aspect is that the non-Hermitian tuple is treated as a
complete self-adjoint pencil at a time, rather than regarding its possibly
dependent Hermitian coordinates as an independent family.  The Gaussian
comparison matches the full real covariance in~\eqref{eq:real-covariance}, so
the pseudo-variance
\(m_\kappa=\E[(\xi^{(\kappa)})^2]\) is retained at finite \(n\);
its contribution remains visible in~\eqref{eq:exact-variance}.  Only after
Proposition~\ref{prop:gaussian-limit} identifies the limit of the matching
Gaussian matrices does the circular structure emerge.  The
bounded-entry comparison and the finite-fourth-moment transfer also play
distinct roles: the latter uses the Bai--Yin bound to control each fixed
centered tail law, with \(n\to\infty\) taken before the truncation level tends
to infinity.

Several quantifiers in Theorem~\ref{thm:main} are deliberately fixed.  The number of
colors \(d\), the tested polynomial (its degree and coefficients), and the
matrix-amplification size \(q\) do not vary with \(n\).  The theorem therefore
does not give uniform control for \(d=d_n\), \(q=q_n\), namely, growing degrees or
\(n\)-dependent coefficients.  Likewise, the model does not include
\(n\)-dependent deterministic matrix backgrounds or correlations among the
entry arrays.
Such extensions would require additional assumptions that preserve both the
comparison step and the identification of the limiting family.  The
almost-sure formulation also uses the nested-corner coupling specified in
Definition~\ref{def:iid-entry-model}; it is not an assertion that is
automatically invariant under arbitrary recouplings across \(n\).

For nonnormal polynomials, the results on the spectrum are intentionally stated
in terms of singular values and invertibility: strong \(*\)-convergence alone
does not imply convergence of eigenvalue distributions or Brown measures in
this nonnormal setting.  Establishing such results would
require additional estimates on small singular values, together with the
uniform integrability needed for logarithmic potentials.  These ingredients
are not supplied by the norm and moment comparison developed here.  The
argument is also qualitative: it proves almost-sure convergence for every
fixed test but does not state a quantitative convergence rate.  Finally, the
fourth-moment assumption is sharp for a theorem uniform over the full model
class considered here, but this does not rule out weaker assumptions for
particular polynomials, weaker modes of convergence, or more specialized
ensembles.

These issues pose three natural directions: strong convergence in
the presence of suitable deterministic matrix backgrounds, quantitative
control for families of tests whose complexity grows with \(n\), and
Brown-measure questions supported by uniform small-singular-value estimates.
Each direction requires a technique not supplied by the present proof.  
The role of the current theorem is to isolate a setting in which covariance matching,
Gaussian strong convergence, and fixed-level fourth-moment truncation can be
combined without imposing circular symmetry on the entries.

\bibliographystyle{alpha}
\bibliography{reference}

\end{document}